\documentclass[11pt,onecolumn]{article}    
\usepackage{geometry}
\usepackage{graphicx}          

\usepackage{amsmath,mathrsfs,amsthm}
\usepackage{amssymb}
\usepackage{pstricks,pst-plot,psfrag}
\usepackage[all]{xy}
\usepackage{graphicx,subfigure,xspace,bm}

\usepackage{amsmath,mathrsfs, textcomp, eufrak, mathtools}
\usepackage{amssymb,enumerate}
\usepackage{pstricks,pst-plot,psfrag}
\usepackage[all]{xy}
\usepackage{graphicx,subfigure,xspace,bm}
\usepackage{etex}
\usepackage{todonotes}
\usepackage[normalem]{ulem}
\usepackage{notoccite}
\usepackage{mathtools}

\usepackage{authblk}

\newcommand{\real}{{\mathbb{R}}}
\newcommand{\integers}{\mathbb{Z}}

\newcommand\redsout{\bgroup\markoverwith{\textcolor{red}{\rule[0ex]{2pt}{5pt}}}\ULon}

\newtheorem{theorem}{Theorem}
\newtheorem{definition}{Definition}
\newtheorem{assumption}{Assumption}
\newtheorem{lemma}{Lemma}

\newtheorem{corollary}{Corollary}
\newtheorem{proposition}{Proposition}

\definecolor{darkgreen}{rgb}{0,0.5,0}

\newcommand{\oprocendsymbol}{\hbox{$\bullet$}}
\newcommand{\oprocend}{\relax\ifmmode\else\unskip\hfill\fi\oprocendsymbol}

\newcommand{\eps}{\epsilon}

\newcommand{\map}[3]{#1:#2 \rightarrow #3}

\newcommand\upscr[2]{#1^{\textup{#2}}}

\newcommand{\dout}{\upscr{d}{}}
\newcommand{\doutp}{\upscr{\bar{d}}{}}

\newcommand{\Nin}{\upscr{N}{in}}

\newcommand{\Ninp}{\upscr{\bar{N}}{in}}
\newcommand{\Nout}{\upscr{N}{out}}
\renewcommand{\S}{\mathcal{S}}

\newcommand{\G}{G}
\newcommand{\Gp}{\bar{G}}

\newcommand{\B}{\mathcal{B}}
\newcommand{\A}{\mathbb{A}}

\newcommand{\z}{\mathbf{z}}
\DeclareMathOperator*{\argmin}{arg\,min}

\newcommand{\x}{\mathbf{x}}
\newcommand{\w}{\mathbf{w}}
\newcommand{\g}{\mathbf{g}}
\renewcommand{\Pr}{\mathbb{P}}

\newcommand{\Es}{\mathcal{E}}
\newcommand{\Esp}{\mathcal{\bar{E}}}
\newcommand{\Ex}{\mathbb{E}}
\newcommand{\V}{\mathcal{V}}
\newcommand{\Q}{\mathbb{Q}}

\usepackage{cancel}

\definecolor{BBlue}{cmyk}{.98,0.10,0,.25}

\title{Rate of Convergence for Distributed Optimization with Uncertain Communications\footnote{This note is primarily written as a research announcement on arXiv; a complete version, combined with an extended version of~\cite{PR-BG-TL-BT:18-necsys}, will appear elsewhere.}}
\author[1]{Pouya Rezaeinia}   
\author[2]{Bahman Gharesifard}      
\affil[1]{Sauder School of Business, University of British Columbia}
\affil[2]{Department of Mathematics and Statistics, Queen's University}

\begin{document}
	
\maketitle

		\begin{abstract}                          
			We consider the distributed optimization problem for the sum of convex functions where the underlying communications network connecting agents at each time is drawn at random from a collection of directed graphs. Building on an earlier work~\cite{PR-BG-TL-BT:18-necsys}, where a modified version of the subgradient-push algorithm is shown to be almost surely convergent to an optimizer on  sequences of random directed graphs, we find an upper bound of the order of $\sim O(\frac{1}{\sqrt{t}})$ on the convergence rate of our proposed algorithm, establishing the first convergence bound in such random settings.
		\end{abstract}

	\section{Introduction}
	Distributed optimization of a sum of convex functions is concerned with solving the following problem: a network of nodes $\V = \{v_1,v_2,...,v_n\}$, each having access to a private local convex function ${f_i: \real^d \rightarrow \real}$, aim to solve the minimization problem 
	\begin{align}\label{minimization}
	\text{minimize } F(\z) \triangleq \sum_{i=1}^{n} f_i(\z), \quad \z\in \real^d,
	\end{align}
	in a \emph{distributed manner}, i.e., only by exchanging limited information on their estimate of the optimizer. We consider a communication layer between the nodes and at each time $ t \geq 0 $, nodes can exchange information about their \emph{state}, that is their local estimate of a solution of~\eqref{minimization}. We are seeking for an algorithm where each node utilizes the private local convex function along with the estimates received from its neighbours to solve~\eqref{minimization}. The importance of the problem stems from its applications in a variety of contexts
	, ranging from sensor localization~\cite{JWD-AF-FB:10} and distributed electricity generation and smart grids~\cite{ADDG-STC-CNH:12}, \cite{LG-UT-SHL:12} to statistical learning~\cite{AN-AO-CAU:17-bayesian}.

	The main purpose of this note is to complete a missing part of an earlier work~\cite{PR-BG-TL-BT:18-necsys}, by providing some convergence rate results, and for this reason, we avoid surveying the related literature in details. A complete and combined version of this work along with~\cite{PR-BG-TL-BT:18-necsys} will appear elsewhere. 
		

	\subsection{Mathematical Preliminaries}\label{sec:prelim}
	Let $ \real $ denote the set of real numbers, and let $ \real_{\geq 0} $ and $ \integers_{\geq 0} $ denote the set of non-negative real numbers and integers, respectively. For a set $\A$, we write $S\subset \A$ if $S$ is a proper subset of $\A$, and we call the empty set and $\A$ trivial subsets of $\A$. The complement of $S$ is denoted by $S^c$. Let $ |S| $ denote the cardinality of a finite set $ S $. We view all vectors in $\real^n$ as column vectors, where $n$ is a positive integer. We denote by $ \|\cdot\| $ and $ \|\cdot\|_1 $, the standard Euclidean norm and the $1$-norm on $ \real^n $, respectively. The notation $A'$ and $v'$ will refer to the transpose of the matrix $A$ and the vector $v$, respectively. We use $ \real_{\geq 0}^{n\times n}$ to denote the set of $n\times n$ non-negative real-valued matrices. A matrix $A\in \real_{\geq 0}^{n\times n}$ is column-stochastic if each of its columns sums to 1; row-stochastic matrices are defined similarly, and when both conditions are satisfied, we refer to $ A $ as doubly stochastic. For a given $A\in \real_{\geq 0}^{n\times n}$ and any nontrivial $S\subset [n]\triangleq \{1,,\ldots,n\}$, we define ${A_{SS^c} \triangleq \sum_{i\in S, j\in S^c}A_{ij}}$. 
	\subsubsection{Graph theory}
	A (weighted) \emph{directed graph} $\G\triangleq (\V,\Es,W) $ consists of a node set $\V \triangleq \{v_1,v_2,\ldots,v_n\}$, an edge set $ {\Es \subseteq \V\times \V}$, and a weighted \emph{adjacency matrix} $ {W \in \real^{n\times n}_{\geq0}} $, with ${ W_{ji}>0 }$ if and only if $ (v_i,v_j)\in \Es $, in which case we say that $v_i$ is connected to $v_j$. Similarly, given a matrix $ W \in \real^{n\times n}_{\geq0} $, one can associate to $ W $ a directed graph $ \G=(\V,\Es) $, where $ (v_i,v_j)\in \Es $ if and only if $ {W_{ji}>0 }$, and hence $ W $  is the corresponding weighted adjacency matrix for $ \G $. The in-neighbors and the out-neighbors of $v_i$ are the sets of nodes ${\Nin_i= \{j\in [n]: W_{ij}>0\}}$ and ${\Nout_i= \{j\in [n]: W_{ji}>0\}}$, respectively. The out-degree of $v_i$ is ${\dout_i=|\Nout_i|}$; we simply drop the in and out indices when the graph is undirected.  In the directed graph $\G=(\V,\Es,W) $, a path is a sequence of distinct nodes $v_{i_1},\ldots,v_{i_k}$ for some ${k\in [n]}$ such that $(v_{i_j},v_{i_{j+1}})\in \Es$ for all ${j\in [k-1] }$. A directed graph is \emph{strongly connected} if there is a path between any pair of nodes. If the directed graph ${\G=(\V,\Es,W)} $ is strongly connected, we say that $W$ is irreducible. For graphs $\G_1=(\V,\Es_1) $ and $\G_2=(\V,\Es_2)$ on the node set $\V$, $\G = \G_1 \cup \G_2$ is the graph on the node set $\V$ with the edge set $\Es = \Es_1\cup\Es_2$.
	\subsubsection{Sequences of random column-stochastic matrices}
	Let $\S_{n}$ be the set of $n\times n$ column-stochastic matrices, and let $\mathcal{F}_{\S_{n}}$ denote the Borel $\sigma$-algebra on $\S_{n}$. Given a probability space $(\Omega, \B, \mu)$, a measurable function $ {\map{W}{(\Omega, \B, \mu)}{(\S_n, \mathcal{F}_{\S_{n}})}} $ is called a random column-stochastic matrix, and a sequence $ \{W(t)\} $ of such measurable functions on $(\Omega, \B, \mu)$ is called a random column-stochastic matrix sequence; throughout, we assume that $t\in \integers_{\geq 0}$. Note that for any $ \omega \in \Omega $, one can associate a sequence of directed graphs $\{\G(t)(\omega)\}$ to $\{W(t)(\omega)\} $, where $(v_i,v_j)\in \Es(t)(\omega)$ if and only if $W_{ji}(t)(\omega)>0$. This in turn defines a sequence of random directed graphs on $ \V=\{v_1,\ldots, v_n\} $, which we denote by $\{\G(t)\}$.

\subsection{Push-sum Algorithm}
	
	Let us present the push-sum dynamics, introduced in~\cite{DK-AD-JG:03}, formally: consider a network of nodes $ \V=\{v_1,v_2,\ldots, v_n\} $, where node $ v_i \in \V$ has an initial state $x_i(0)\in \real$. The push-sum algorithm is defined as follows. Each node $ v_i $ maintains and updates, at each time ${t\geq 0}$, two state variables $ x_i(t) $ and $ y_i(t) $. The first state variable is initialized to $ x_i(0) $ and the second one is initialized to $ y_i(0)=1 $, for all $i\in [n]$. At time $t\geq0$, node $v_i$ sends $\frac{x_i(t)}{\dout_i(t)}$ and $\frac{y_i(t)}{\dout_i(t)}$ to its out-neighbors in the random directed graph $\G(t) = (\V, \Es(t), W(t))$, which we assume to contain self-loops at all nodes for all $t\geq 0$. At time $ (t+1) $, node $ v_i $ updates its state variables according to 
	\begin{align}\label{eqn:main-algo}
	\x_i(t+1) &=\sum_{j\in \Nin_i(t)}\frac{\x_j(t)}{\dout_j(t)},\cr
	y_i(t+1) &= \sum_{j\in \Nin_i(t)}\frac{y_j(t)}{\dout_j(t)},\cr
	\z_i(t+1) &= \frac{\x_i(t+1)}{y_i(t+1)}.
	\end{align}
	Here, $\z_i(t+1) $ is the estimate by node $ v_i $ of the average $ \bar{\x}=\frac{1}{n} \sum_{i=1}^n\x_i(0) $. As a by-product, in~\cite{AN-AO:15-tac}, it is shown that as long as the sequence of graphs satisfies some uniform strong connectivity over time, the state $ z_i(t) $ is guaranteed to converge to $ \bar{\x} $ for all $ i $; in fact, one can obtain convergence rates. In a recent work~\cite{PR-BG-TL-BT:19}, we established an ergodicity criterion for the sequence of column-stochastic matrices corresponding to the push-sum protocol, and demonstrated that a large class of time-varying sequences of random directed graphs satisfies such conditions. Applying this result to a random sequence of graphs generated using a time-varying $B$-irreducible probability matrix, we obtained the first known convergence rates for the push-sum algorithm in random settings. It turns out that pairing the push-sum protocol~\eqref{eqn:main-algo} with subgradient flow leads to a distributed algorithm with guaranteed convergence to a solution of~\eqref{minimization}, when the directed graph is time-varying but satisfies some uniform strong connectivity. It is shown in~\cite[Remark~5]{AN-AO-WS-17} that in the particular case where the graph is undirected, the results are actually extendable to random settings. This being said, extending such results to general random directed graphs, as demonstrated in our previous work~\cite{PR-BG-TL-BT:18-necsys} and also here, is highly non-trivial.
	
	
	
%
%
%

	\section{Statement of the Problem}\label{sec:spa}
	Consider the distributed optimization problem~\eqref{minimization}. 
	Suppose now that the communication layer between nodes at discrete time instances $t\geq 0$ is specified by a sequence of random directed graphs $\{\G(t)\}$. Starting with some initial estimate of an optimal solution, at each time $t$, each node communicates with its neighbors in $\G(t)$: sends its values to its out-neighbors and updates its values according to those of its in-neighbors. We assume that the set of optimal solutions is nonempty. One standing assumption throughout this paper is that each node knows its out-degree at every time, which is shown to be necessary in~\cite{JH-JT:15}. Our main objective in~\cite{PR-BG-TL-BT:18-necsys} was to show that a modified version of the so-called subgradient-push algorithm successfully accomplishes the task of minimizing $F(\z)$ in a distributed fashion, and under the assumption that the communication network is directed and random. This key point distinguishes our work from the existing results in the literature~ \cite{IL-AO:11}, \cite{AN:14-sv}, \cite{AN-AO:15-tac}, \cite{AN-AO:09}, \cite{SSR-AN-VVV:10}. We complete this work by providing convergence rate results.

	\section{Modified Subgradient-Push Algorithm}\label{sec:algo}
	
	In the subgradient-push (SP) algorithm, each node $v_i$ maintains and updates, at each time $t \geq 0$, two vector variables $\x_i(t),\w_i(t)\in \real^d$ as well as a scalar variable $y_i(t)\in \real$. For each $i\in [n]$, the vector $\x_i(0)$ is initialized to the estimate of an optimal solution of node $v_i$ and $y_i(0) $ is initialized to $1$, i.e., $y_i(0)=1$. At time $t \geq 0$, node $v_i$ sends $\x_i(t)$ and $y_i(t)$ to its out-neighbors in the directed graph 
	of the available communication channels $\Gp(t) = (\V,\Esp(t))$, which is assumed to contain self-loops at all nodes. At time $(t + 1)$, node $v_i$ updates its variables according to
	\begin{align}\label{Sub-Push}
	\w_i(t+1) &= \sum_{j\in \Ninp_i(t)} \frac{\x_j(t)}{\doutp_j(t)},\cr
	y_i(t+1) &= \sum_{j\in \Ninp_i(t)} \frac{y_j(t)}{\doutp_j(t)},\cr 
	\z_i(t+1) &= \frac{\w_i(t+1)}{y_i(t+1)},\cr 
	\x_i(t+1) & = \w_i(t+1) - \alpha(t+1)\g_i(t+1),
	\end{align}
	where $\g_i(t+1)$ is a subgradient of the convex function $f_i$ at $\z_i(t+1)$ and 
	$\alpha(t) = \frac{1}{t^{\gamma}}$ for some $0.5<\gamma < 1$. For this choice if $\alpha(t)$ we have 
	\begin{align*}
	\sum_{t=1}^{\infty}\alpha(t) = \infty, \quad\text{ and }\sum_{t=1}^{\infty}\alpha^2(t)<\infty.
	\end{align*}
	At each time $t$, $\z_i(t)$ is the estimate by node $v_i$ of a minimizer of $F(\z)$. For each $i\in [n]$, $\Ninp_i(t)$ is the set of in-neighbors of $v_i$ and $\doutp_i(t)$ is its out-degree in $\Gp(t)$. We assume that all functions $f_i$ are Lipschitz continuous, i.e., for all $i\in [n]$ there exists $L_i$ such that $\|\g_i\|\leq L_i$. For our future analysis we define $L = \sum_{i=1}^{n}L_i$. 
	
	As shown in for example~\cite{IL-AO:11}, in undirected settings, or directed settings with doubly stochastic weighting, the same algorithm can be used along sample paths and that leads to a convergence result in expectation. Now, one hopes to prove a similar result for~\eqref{Sub-Push} in random settings, with an important caveat that this algorithm heavily relies on division of two correlated random variables. More importantly, the proof of the convergence in deterministic settings heavily relies on bounding the denominator away from zero. This issue makes it impossible to carry the analysis of~\cite{PR-BG-TL-BT:19} to random setting. The main novelty of this work is to modify this algorithm, while still using its useful features; this is what we describe next.

	\subsection{The Modified Subgradient-Push Algorithm}\label{sec:MSP}
	Here we introduce the modified subgradient-push algorithm (MSP). In the MSP algorithm, each node ${v_i\in \V} $ sends its values to its out-neighbors in $\Gp(t)$ only if ${y_i(t)\geq \frac{1}{n^{2n}}}$. In this sense, at time $ t $ node $ v_i $ only receives information from the subset $ \Nin_i(t)\subseteq \Ninp_i(t) $ of its in-neighbors given by   
	\[
	\Nin_i(t)=\Ninp_i(t)\backslash \{ v_j \ | \ y_j(t)< \frac{1}{n^{2n}}\}.
	\]
	Indeed, this construction induces an effective communication network graph \sloppy ${\G(t) = (\V,\Es(t),W(t))}$ at time $t$ with the same set of nodes as $ \Gp(t) $ and the set of edges $\Es(t) \subseteq \Esp(t)$. 
	Similar to the SP algorithm, the MSP algorithm is given as 
	\begin{align}\label{M-Sub-Push}
	\w_i(t+1) &= \sum_{j\in \Nin_i(t)} \frac{\x_j(t)}{\dout_j(t)},\cr
	y_i(t+1) &= \sum_{j\in \Nin_i(t)} \frac{y_j(t)}{\dout_j(t)},\cr 
	\z_i(t+1) &= \frac{\w_i(t+1)}{y_i(t+1)},\cr 
	\x_i(t+1) & = \w_i(t+1) - \alpha(t+1)\g_i(t+1),
	\end{align}
	where $ \Ninp_i(t) $ is replaced with $ \Nin_i(t) $ for all $ i \in [n] $ and $ t \geq 0 $. Here the $\dout_i(t)$ are the out-degrees of nodes in the effective communication network graph $\G(t)$. Throughout this article, whenever we write $ \G(t) $, or $ \Nin_i(t) $, we refer to the subgraph of $ \Gp(t) $ resulting from this modification.

	\section{Random Setting and Main Result}
	
	We next state our assumptions on the sequence of random graphs; as we mentioned in Section~\ref{sec:MSP}, we denote by $\Gp(t) = (\V,\Esp(t))$ the graph of the available communication channels and by ${\G(t) = (\V,\Es(t),W(t))}$ the communication network graph utilized by the nodes at time~$t$. 
	
	\begin{assumption}\label{assumption:1}
		Let $\mathcal{G} = \{\Gp_1,\Gp_2,\ldots,\Gp_{2^{n^2-n}}\}$ be the set of all the possible graphs of available communication channels on the node set \sloppy ${\V}$ with self-loops at all nodes. We assume that the sequence of realized available communication graphs $ \{\Gp(t)\}$ satisfies:
		\begin{enumerate}
			\item at each time ${t\geq 0}$, the corresponding graph for communications is $ \Gp_b$ with probability ${\Pr(\Gp(t) = \Gp_b) = p_b}$, where $b\in [2^{n^2-n}]$;
			\item $\bigcup_{b: p_b > 0} \Gp_b$ is strongly connected;
			\item the graphs $\Gp(t)$ are independent of each other.
		\end{enumerate}
	\end{assumption}
	
	This assumption states that the graphs of available communication channels are drawn independently and from the set of all possible graphs on the node set $\V$ with self-loops at all nodes. Note that part (ii) imposes a mild connectivity assumption on this sequence of graphs, which is analogous to deterministic settings. 

	In~\cite{PR-BG-TL-BT:18-necsys}, it is shown that the proposed algorithm, namely the MSP algorithm, converges asymptotically to an optimal solution for our random setting, as stated next. 	
		\begin{theorem}\cite{PR-BG-TL-BT:18-necsys}\label{thm:convergence}
		Consider the MSP algorithm~\eqref{M-Sub-Push} and suppose that the sequence of available communication channels $\{\Gp(t)\}$ satisfies Assumption 1. Then we have 
			\begin{align*}
			\lim_{t\rightarrow \infty} \z_i(t) = \z^*, \quad\text{for all }i\in [n],
			\end{align*}
			almost surely, where $\z^*$ is an optimal solution of~\eqref{minimization}. 
	\end{theorem}
	
	The main result of this work is obtaining convergence rates, as we state next. 
	
	\begin{theorem}\label{thm:rate}
		Consider the MSP algorithm~\eqref{M-Sub-Push} and suppose that the sequence of available communication channels $\{\Gp(t)\}$ satisfies Assumption 1. Define $S(0)=0$ and $S(t) = \sum_{s=0}^{t-1}\alpha(s+1)$ for all $t\geq 1$. Assume that every node $i$ maintains the vector $\tilde{\z}_i(t)\in \real^d$ initialized with any $\tilde{\z}_i(0)\in \real^d$ with the following update rule:
			\begin{align*}
			\tilde{\z}_i(t+1) = \frac{\alpha(t+1)\z_i(t+1) + S(t)\tilde{\z}_i(t)}{S(t+1)},\quad \forall t\geq 0.
			\end{align*}
			Then, we have for all $t\geq 1$, $i\in[n]$, and any ${\z^{*} \in Z^{*}}$,
			\begin{align*}
			\Ex&\left[F(\tilde{\z}_i(t+1))-F(\z^*)\right]
			\\&\leq
			\Gamma\left[\frac{n\|\bar{\x}(0) - \z^*\|_1}{2} \right.
			+ \left(1 + \frac{1}{2\gamma-1}\right)\frac{L^2}{2n} 
			\\ &+\frac{60 L\sum_{j=1}^{n}\|x_j(0)\|_1}{\delta(1-\lambda)}
			\left.+\frac{60dL^2}{\delta(1-\lambda)}\left(1+\frac{1}{2\gamma-1}\right)\right],
			\end{align*}
			where $\Gamma = \frac{(1-\gamma)}{(t+2)^{1-\gamma}-1}$.
	\end{theorem}
	The rest of this paper is dedicated to the proof of this result. 
	\section{Convergence Rate Analysis}
	We start by defining a perturbed version of the push-sum algorithm subject to the modification that we have made. In this section, for most parts, we work with scalar variables.
	\subsection{Modified Perturbed-Push Algorithm}
	Here we describe the modified perturbed-push (MPP) algorithm. In this algorithm, each node $v_i$ maintains and updates scalar variables $x_i(t),w_i(t)$ and $y_i(t)$, where the $x_i(0)$ are arbitrary scalars and the $y_i(0) $ are initialized to $1$, for all $i\in [n]$. The update rule at time $t+1$ is  
	\begin{align}\label{M-Per-Push}
	w_i(t+1) &= \sum_{j\in \Nin_i(t)} \frac{x_j(t)}{\dout_j(t)},\cr
	y_i(t+1) &= \sum_{j\in \Nin_i(t)} \frac{y_j(t)}{\dout_j(t)},\cr 
	z_i(t+1) &= \frac{w_i(t+1)}{y_i(t+1)},\cr 
	x_i(t+1) & = w_i(t+1) + \eps_i(t+1),
	\end{align}
	where the $\eps_i(t)$ are perturbations at time $t$, to be specified later. Similar to the MSP algorithm, node $v_i$ shares its values only if $ y_i(t)\geq \frac{1}{n^{2n}}$, and the $\dout_i(t)$ and the $\Nin_i(t)$ are the out-degrees and the in-neighbors of the nodes in the effective communication network graph $\G(t)$, respectively. 
	
	To write the algorithm in a matrix form,  for all $t\geq 0$, we let the column-stochastic matrix $W(t)$ be the weighted adjacency matrix associated with the effective communication network graph $\G(t)$ with entries 
	\begin{equation}\label{eqn:W}
	W_{ij}(t) = 
	\begin{cases} \frac{1}{\dout_j(t)} &\text{if } j\in \Nin_i(t),\\
	0&\text{otherwise.} 
	\end{cases}
	\end{equation}
	Writing the MPP algorithm in matrix form, we have for all $t\geq0$
	\begin{align}\label{alg}
	w(t+1) &= W(t)x(t),\cr
	y(t+1) &= W(t)y(t),\cr 
	z_i(t+1) &= \frac{w_i(t+1)}{y_i(t+1)},\quad\text{for all }i\in [n]\cr 
	x(t+1) & = w(t+1) + \eps(t+1),
	\end{align}
	where $\eps(t) = (\eps_1(t),\ldots,\eps_n(t))'$ is the vector of perturbations at time $t$. Here, $ w(t)= (w_1(t), \ldots, w_n(t))\in \real^n $, $ z(t)=(z_1(t), \ldots, z_n(t))\in \real^n $, and in this sense we have treated each component of the individual agent's state variables separately. In the analysis of the MPP algorithm, we assume that $\|\eps(t)\|_1\leq \frac{U}{t^\gamma}$. This assumption holds when we regard the subgradient term in the MSP algorithm as perturbation. Throughout, the $W(t)$ denote the adjacency matrices associated with the effective communication graphs defined in~\eqref{eqn:W}.
	
	\subsection{Convergence of the MPP algorithm}
	We study the convergence properties of MPP algorithm.We first tackle the issue of connectivity of the sequence of matrices that are generated through the MPP algorithm. One of the key properties that we need to ensure for the sequence of random matrices induced by the MPP algorithm is  the directed infinite flow property, which we recall next. 
	
	\begin{definition}~\cite[Definition 3]{PR-BG-TL-BT:19}
		We say that a sequence of matrices $\{W(t)\}$ has the directed infinite flow property if for any non-trivial ${S \subset [n]}$ $$ \sum_{t=0}^{\infty} W_{SS^c}(t) = \infty, \quad\text{almost surely.}$$
	\end{definition}

	The following proposition presents an upper bound on how well the sequences $z_i(t+1)$ estimate the average $\bar{x}(t) := \frac{1^{\prime} x(t)}{n}$ for each sample path, when the sequence of matrices $\{W(t)\}$ has the directed infinite flow property. This will allow us to state the first connectivity result in a random setting, previously provided in~\cite{PR-BG-TL-BT:18-necsys}.
	
	\begin{proposition}\cite{PR-BG-TL-BT:18-necsys}\label{prop:1}
		Consider the MPP algorithm~\eqref{alg} and suppose that the sequence $\{W(t)\}$ has the directed infinite flow property, almost surely. Then, we have
		\begin{align}\label{eqn:Thm1}
		| z_i&(t+1) - \bar{x}(t) | 
		\cr&\leq\frac{2}{y_i(t+1)}\left( \Lambda_{t,0} \|x(0)\|_1+ \sum_{s=1}^{t} \Lambda_{t,s} \|\eps(s)\|_1 \right),
		\end{align}
		almost surely, where $\Lambda_{t,s}\in (0,1)$ for all $t\geq s\geq 0$.
	\end{proposition}

	Next, we state two lemmas on the connectivity of the sequence of matrices $\{W(t)\}$ given our random setting, followed by a useful corollary. These results are  provided in~\cite{PR-BG-TL-BT:18-necsys}. 
	\begin{lemma}\cite{PR-BG-TL-BT:18-necsys}\label{connectivity}
		Consider the MPP algorithm~\eqref{alg} and let $ W(t) $ given in~\eqref{eqn:W} be the weighted adjacency matrix of the graph induced by the available communication channel at time $ t $, and suppose that Assumption~\ref{assumption:1} holds for the sequence $W(t)$. Then,
		\begin{enumerate}
			\item for all $t\geq 0$, we have\[ \Pr(W(t+2n-3:t) \text{ irreducible})\geq p>0,\]
			where $p=(\min_{b:p_b>0} p_b)^{2n-2}$.
			\item $\{W(t)\}$ has the directed infinite flow property. 
		\end{enumerate}
	\end{lemma}

	\begin{corollary}\cite{PR-BG-TL-BT:18-necsys}\label{Cor:bnd}
		Consider the MPP algorithm~\eqref{alg} and let $ W(t) $ given in~\eqref{eqn:W} be the weighted adjacency matrix of the graph induced by the available communication channel at time $ t $, and suppose that Assumption~1 holds. We have 
		\begin{align}\label{eqn:Cor}
		\left| z_i(t+1) - \bar{x}(t) \right| \leq \frac{2}{\delta} \left( \Lambda_{t,0} \|x(0)\|_1 + \sum_{s=1}^{t} \Lambda_{t,s} \|\eps(s)\|_1 \right).
		\end{align}
		almost surely, where $\Lambda_{t,0} \in (0,1)$ for all $t\geq s\geq 0$ and ${\delta = \frac{1}{n^{2n}}}$.
	\end{corollary}

	\begin{lemma}\cite{PR-BG-TL-BT:18-necsys}\label{lem:prodrate}
		Consider the MPP algorithm~\eqref{alg} and let $ W(t) $ given in~\eqref{eqn:W} be the weighted adjacency matrix of the graph induced by the available communication channel at time $ t $, and suppose that Assumption~1 holds. Then
		
		\begin{enumerate}
			\item for all $t\geq s\geq 0$ we have
			\begin{align}\label{ineq:Lamlam}
			\Pr\left(\Lambda_{t,s} > 2\lambda^{t-s}\right) \leq 13 e^{-c_1(t-s)},\text{and }\quad  \Ex[\Lambda_{t, s}]\leq 15 \lambda^{t-s},
			\end{align}
			where $\lambda = \left(1-\frac{1}{n^{\frac{4nB}{p}}}\right)^{ \frac{p}{2nB}} \in (0,1)$, $c_1 = \frac{p^2}{4B}$ and ${B=2n-2}$;
			\item we have that
			\begin{align*}
			\Pr(\lim_{t\rightarrow \infty} \Lambda_{t,s} =0)=1.
			\end{align*}	
		\end{enumerate}
	\end{lemma}

	In the following lemma, by assuming an upper bound on the perturbations we present our first convergence rate result. Note that the first two parts of the following lemma are borrowed from 
	
	\begin{lemma}\label{lem:conv}
		Consider the MPP algorithm~\eqref{alg} and let $ W(t) $ given in~\eqref{eqn:W} be the weighted adjacency matrix of the graph induced by the available communication channel at time $ t $, and suppose that Assumption~1 holds. In addition, assume that the perturbations $\epsilon_{i}(t)$ are bounded as follows: 
		\begin{align*}
		\|\eps(t)\|_1\leq \frac{U}{t^{\gamma}}, \quad \text { for all } t \geq 1,
		\end{align*}
		for some scalar $U>0$. Then,  
			\begin{align*}
			\Ex \left[{\frac{1}{\sum_{k=0}^{t} \alpha(k+1)} \sum_{k=0}^{t} \alpha(k+1)\left|z_{i}(k+1)-\overline{x}(k)\right|} \right]&\\ \leq \frac{2c_2\Gamma}{\delta(1-\lambda)}\cdot\left( \|x(0)\|_{1}+U\left(1+\frac{1}{2\gamma-1}\right)\right)&,
			\end{align*}
			where $\Gamma $ is defined in Theorem~\ref{thm:rate}.
	\end{lemma}
	
	\begin{proof}
	Using Corollary~\ref{Cor:bnd} we have
		\begin{align*}
		\sum_{k=1}^{t} &\alpha(k+1) \left|z_{i}(k+1)-\overline{x}(k)\right| \\\leq& \sum_{k=1}^{t} \frac{1}{(k+1)^\gamma}\frac{2}{\delta} \left( \Lambda_{k,0} \|x(0)\|_1 + \sum_{s=1}^{k} \Lambda_{k,s}\|\eps(s)\|_1 \right) \\ 
		=& \frac{2\|x(0)\|_1}{\delta}\sum_{k=1}^{t} \frac{1}{(k+1)^\gamma} \Lambda_{k,0}\\
		& + \frac{2}{\delta}\sum_{k=1}^{t} \sum_{s=1}^{k} \frac{1}{(k+1)^\gamma}\Lambda_{k,s}\|\eps(s)\|_1\\ 
		\leq & \frac{2\|x(0)\|_1}{\delta}\sum_{k=1}^{t}  \Lambda_{k,0} + \frac{2U}{\delta}\sum_{k=1}^{t} \sum_{s=1}^{k} \frac{1}{s^{2\gamma}}\Lambda_{k,s}.
		\end{align*}	
		Taking expectations from both sides, using~\eqref{ineq:Lamlam}, we have 
		\begin{align*}
		\Ex&\left[\sum_{k=1}^{t} \alpha(k+1) \left|z_{i}(k+1)-\overline{x}(k)\right| \right]
		\\
		&\leq  \Ex\left[\frac{2\|x(0)\|_1}{\delta}\sum_{k=1}^{t}  \Lambda_{k,0} + \frac{2U}{\delta}\sum_{k=1}^{t} \sum_{s=1}^{k} \frac{1}{s^{2\gamma}}\Lambda_{k,s}\right]\\
		&\leq \frac{30\|x(0)\|_1}{\delta} \sum_{k=1}^{t}\lambda^k + \frac{30U}{\delta}\sum_{k=1}^{t} \sum_{s=1}^{k} \frac{1}{s^{2\gamma}}\lambda^{k-s}\\ 
		&\leq \frac{30\|x(0)\|_1}{\delta} \frac{\lambda}{1-\lambda} + \frac{30U}{\delta}\sum_{k=1}^{t} \sum_{s=1}^{k} \frac{1}{s^{2\gamma}}\lambda^{k-s}.
		\end{align*}
		For the second term on the right-hand side we have
		\begin{align*}
		\sum_{k=1}^{t} \sum_{s=1}^{k} \frac{1}{s^{2\gamma}}\lambda^{k-s} = \sum_{s=1}^{t}\frac{1}{s^{2\gamma}}\sum_{k=s}^{t}\lambda^{k-s} 
		&\leq \sum_{s=1}^{t}\frac{1}{s^{2\gamma}} \frac{1}{1-\lambda}\\
		&\leq  \frac{1}{1-\lambda}\left(1+ \int_{1}^{\infty}\frac{d u}{u^{2\gamma}}\right)\\ 
		&= \frac{1}{1-\lambda} \left(1+ \frac{1}{2\gamma-1}\right) .
		\end{align*}
		Following similar steps as in~\cite[Corollary 3]{AN-AO:15-tac}
		\begin{align}\label{ineq:expsum}
		\Ex&\left[\sum_{k=0}^{t} \alpha(k+1) \left|z_{i}(k+1)-\overline{x}(k)\right| \right] \cr 
		&\quad \leq \frac{30}{\delta(1-\lambda)}\left(\|x(0)\|_1 + U \left(1+ \frac{1}{2\gamma-1}\right) \right).
		\end{align}
		In addition, we have the following inequality 
		\begin{align}\label{ineq:alphasum}
		\sum_{k=0}^{t} \alpha(k+1) = \sum_{k=0}^{t} \frac{1}{(k+1)^\gamma}
		\geq \int_{0}^{t} \frac{dk}{(k+2)^\gamma}
		=  \frac{(t+2)^{1-\gamma}-1}{1-\gamma} \triangleq  \Gamma^{-1}
		\end{align}
		Therefore, by~\eqref{ineq:expsum} and~\eqref{ineq:alphasum}, we obtain our desired result.
	\end{proof}
	
	\begin{proof} {\bf (Theorem~\ref{thm:rate})}
		Given our presented results, in particular, Lemma~\ref{lem:conv} part (iii), the proof is similar to the lines of the proof of Theorem 2 in~\cite{AN-AO:15-tac}. 
	\end{proof}

	\bibliographystyle{plain}        
	\bibliography{autosam,alias,PR-add,BG-add,Main,Main-add,JC,BG} 

\end{document}